\documentclass{arx}
\usepackage{amsmath}
  \usepackage{paralist}
  \usepackage{graphics} 
  \usepackage{epsfig} 
 \usepackage[colorlinks=true]{hyperref}
\hypersetup{urlcolor=blue, citecolor=red}

\newtheorem{theorem}{Theorem}[section]
\newtheorem{corollary}{Corollary}

\newtheorem{lemma}[theorem]{Lemma}
\newtheorem{proposition}{Proposition}

\theoremstyle{definition}

\newtheorem{remark}{Remark}

\newcommand{\ep}{\varepsilon}
\newcommand{\eps}[1]{{#1}_{\varepsilon}}

\newcommand{\R}{{\Bbb R}}
\newcommand{\C}{{\Bbb C}}
\newcommand{\Z}{{\Bbb Z}}

\title[Pushed traveling fronts in monostable  delayed equations ]
      {Pushed traveling fronts in monostable  equations with 
monotone delayed reaction}

\author[Elena Trofimchuk, Manuel Pinto  and Sergei Trofimchuk]{}

\subjclass{Primary: 34K12, 35K57; Secondary: 92D25.}
 \keywords{Upper and lower
solutions,  monotone traveling waves, pushed fronts,  
asymptotic integration,  minimal speed.}

 \email{trofimch@imath.kiev.ua}
 \email{pintoj@uchile.cl}
 \email{trofimch@inst-mat.utalca.cl}

\begin{document}
\maketitle

\centerline{\scshape Elena Trofimchuk }
\medskip
{\footnotesize
 \centerline{Department of Differential Equations, National Technical University}
   \centerline{Kyiv, Ukraine}
} 

\medskip

\centerline{\scshape Manuel Pinto}
\medskip
{\footnotesize
 \centerline{ Departamento de Matem\'aticas, Facultad de Ciencias,
Universidad de Chile}
   \centerline{Casilla 653, Santiago, Chile}
}

\medskip

\centerline{\scshape Sergei Trofimchuk}
\medskip
{\footnotesize
 \centerline{ Instituto de M\'atematica y F\'isica, Universidad de Talca}
   \centerline{Casilla 747, Talca, Chile}
}

\bigskip

\begin{abstract} We study the existence and uniqueness of wavefronts 
to the scalar reaction-diffusion equations $u_{t}(t,x) = \Delta u(t,x) - u(t,x) + g(u(t-h,x)),$ with monotone delayed reaction term $g: \R_+ \to \R_+$ and $h >0$.   We are mostly interested 
in the situation when the graph of  $g$ is not dominated by  its tangent line at zero, i.e. when the condition   $g(x) \leq g'(0)x,$ $x \geq 0$, is not satisfied. It is well known that, in such a case,  a special type of rapidly decreasing wavefronts (pushed fronts) can appear in  non-delayed equations (i.e. with $h=0$). One of our main goals here  is to  establish a similar result for $h>0$. We prove  the existence of the minimal speed of propagation, the uniqueness of wavefronts (up to a translation) and describe their asymptotics at $-\infty$.  We also present a new uniqueness result for a class of nonlocal lattice equations.
 
\end{abstract}

\section{Introduction}
\noindent In this work, we focus our efforts on the study of the existence, uniqueness and asymptotics  of positive monotone bounded  traveling  wave solutions $u(t,x) = \phi(\nu\cdot x +ct), \ \phi(-\infty) =0,$ to the scalar reaction-diffusion equation
\begin{equation} \label{1}
u_{t}(t,x) = \Delta u(t,x) - u(t,x) + g(u(t-h,x)), \ x \in \R^m. 
\end{equation}
It is assumed  that $\nu \in \R^m, |\nu| =1$, that the wave velocity $c$ is positive and the continuous monotone   nonlinearity $g: \mathbb R_+ \to \mathbb R_+$ satisfies the following assumption  \vspace{2mm}

\noindent  {\rm \bf(H)}  $g$ is strictly increasing and the equation $g(x)= x$ has exactly two nonnegative solutions:
$0$ and $\kappa >0$. Moreover, $g$ is differentiable at the equilibria  with $g'(0) >1,$ $g'(\kappa) < 1,$ and $g$ is $C^1$-smooth in some neighborhood of $\kappa$.  In addition, there exist $C >0,\ \theta \in (0,1],\ \delta >0$ such that   
\begin{equation}\label{gco}
\left|g(u)/u- g'(0)\right| \leq Cu^\theta, \quad u\in  (0,\delta].
\end{equation}
Perhaps, model (\ref{1}) is one of the simplest and most studied monostable delayed reaction-diffusion equations. See \cite{AGT,bn,FZ,FT,ma,ma1,MeiI,TTT,TT,TAT,wlr2,wz} and references therein for more detail regarding 
(\ref{1}) and its non-local versions.  In fact,  the last decade of studies has lead to almost complete description of  the existence, uniqueness and stability properties of wavefronts to (\ref{1}) whenever $g$ satisfies  (\textbf{H}) and  the following quite important sub-tangency  condition
\begin{equation}\label{st}
g(x) \leq g'(0)x, \quad x \geq 0. 
\end{equation}
The latter inequality was already used in the celebrated work \cite{kpp} by   A. Kolmogorov, I. Petrovskii and N. Piskunov, 
where it was assumed that $g'(x) < g'(0)$ for all $x \in (0,\kappa]$. Roughly speaking, inequality (\ref{st})
 amounts to the dominance of the `linear component' within essentially non-linear model (\ref{1}).  It is needless to say that, from the technical point of view, (\ref{st}) allows to simplify enormously the analysis of  traveling waves.  In Subsections 1.1-1.3 below we will illustrate this point in greater detail by discussing  such key issues as the minimal (critical) speed of propagation, the stability, existence and uniqueness of waves, the asymptotic properties of  wave profiles.   Therefore it is not a big surprise that none of these issues has been adequately addressed in that strongly nonlinear case when (\ref{st}) does not hold and  $h>0$\footnote{If $h=0$,  the wavefront problem for (\ref{1}) is essentially bi-dimensional in many aspects and it is rather well understood, cf. \cite{BD,ES,GK,HR,ST,Xin}. Next, since $-u+ g(v)$ is negative  for some $u,v \geq 0$, Schaaf's  results \cite{sch} can not be applied to (\ref{1}). In any event, the question  of pushed waves was not considered in  \cite{sch}.}. So our main objective in this paper is  to complete the study of the existence, uniqueness and asymptotics  of wavefronts to delayed reaction-diffusion equation (\ref{1}) considered  under hypothesis {\rm \bf(H)} and without  condition (\ref{st}). 

At this stage of discussion, it is instructive to raise the same questions but for a different family 
of delayed evolution equations
\begin{equation} \label{1d}
u_{t}(t,x) = [u(t,x+1) + u(t,x-1)-2u(t,x)] - u(t,x) + g(u(t-h,x)), \ x \in \R. 
\end{equation}
It is  obtained from (\ref{1}), $m=1$, by a formal discretization of  the Laplace operator.  Equivalently, we can consider the  lattice differential equations 
\begin{equation} \label{1dl}
u_{n}'(t) = [u_{n+1}(t) + u_{n-1}(t)-2u_n(t)] - u_n(t) + g(u_n(t-h)), \ n \in \Z.
\end{equation}
Equations (\ref{1d}), (\ref{1dl}) are special cases of more general nonlocal lattice population model proposed in \cite{wen}. These equations were analyzed by Ma and Zou in \cite{mazou}.    Once again, in order to prove the existence, uniqueness, monotonicity and stability of wavefronts,   (\ref{st}) together with {\rm \bf(H)} were assumed  in the cited work. 
One of the notorious features   of  \cite{mazou} consists in  its novel (and non-trivial) proof of the wave uniqueness. This proof  does not impose any restriction on   $\sup \{g'(x), x \in [0,\kappa]\}$ what is remarkable in the case of  delayed equations, cf. Subsection 1.2.  

On the other hand, starting from the pioneering work of Zinner, Harris and Hudson \cite{zhh}, significant progress has been achieved in the understanding of waves solutions in non-delayed versions of (\ref{1d}), (\ref{1dl}). See \cite{chen,chfg,mazou,zhh} for more information  and further references.  Non-delayed equation (\ref{1d}) can be also viewed as a particular case of the following differential equation with convolution 
\begin{equation} \label{1dcon}
u_{t}(t,x) = (J* u)(t,x) - u(t,x) + g(u(t,x)), \ x \in \R,
\end{equation}
which was firstly  introduced by Kolmogorov {\it et al} in \cite{kpp}. The latter equation was thoroughly  investigated during the past three  decades using various techniques, see \cite{AGT,CC,co,cdm,KS} and references therein.  Remarkably, sub-tangency condition (\ref{st}) was avoided  in the recent  important contributions \cite{chen, chfg} by  Chen {\it et al.} and \cite{co,cdm} by Coville {\it et al.}  Our present work was nourished in part by several ideas and approaches developed in the mentioned four papers.  For example, our proof of the existence of the minimal speed $c_*$ is also based on the lower-upper solution method. Once again,   the main  difficulty consists in finding  a 'good' upper solution (which additionally has to dominate lower solution), cf. consonant ideas expressed in  \cite[pp. 125-126]{chen} and \cite{TPT}.  As in \cite{chen}, we construct a new formal upper solution (for some velocity $c'$  close to a given velocity $c$) from a given wavefront $\phi(t,c)$. However, in difference with \cite{chen}, our upper solution  is not only formal but also
{\it true upper solution} appearing in pair with an appropriate  lower solution. We neither apply the truncation procedure as in \cite{chen, cdm, zhh} nor we  use  our upper solution as a bound  obligating solutions of associated  truncated problems to converge 
to a true wave solution (this nice idea was proposed in \cite{chen}).  We consider $\phi(t,c)$ only as a skeleton (we call it 'a base function') for creating  a true upper solution  by its suitable  modification.  Recently,  the method of base functions was successfully applied in  our previous work \cite{TPT}  to a model of  the  Belousov-Zhabotinskii reaction. 

Now,  two noteworthy differences appear while comparing (\ref{1dcon}) and (\ref{1}).  First of them is  technical: the presence of the second derivatives in  (\ref{1}) complicates the construction of the lower and upper solutions for (\ref{1})  (these solutions must be $C^1$-smooth or satisfy additional conjugacy relations at the discontinuity points of the derivative , cf.  \cite{BN,CMP,TPT,wz}).  The other difficulty  is more essential: the presence of positive delay $h$ can lead to the non-monotonicity of traveling fronts  \cite{BNPR,FT,NPT,TTT,TT} while such monotonicity seems to be crucial for the applicability of various approaches, e.g. of the sliding solution method \cite{BN,chfg,co,cdm}. 
Precisely in order to avoid front oscillations around $\kappa$, we  will consider strictly increasing $g$ in {\rm \bf(H)}.  It should be mentioned that  monotonicity of $g$ is not obligatory when $h=0$:  this is because function $g(u(t-h))+ku(t)$ is monotone in $u(t)$ for $k \gg 1, \ h=0$, cf. \cite{AGT}.   

Before going back to more detailed analysis of  the main problems addressed in this paper, we would like to state some useful results concerning the wavefronts to equation (\ref{1}) considered under assumption {\rm \bf(H)}.
Set $g'_+:= \sup_{x \geq 0} g(x)/x \geq g'(0) >1$ and define  $c_\#$ [respectively, $c^*$] as this unique positive number $c$ for which  
the characteristic equation 
\begin{equation}\label {che}
\chi(z,c):= z^2 -cz - 1 + pe^{-zch}=0
\end{equation}
with $p = g'(0)$ [respectively, with $p= g'_+$]
has a double positive root. It is easy to see that $c_\# \leq c^*$.  Note that $c_\#= c^*$ coincides with the minimal speed of propagation $c_*$  whenever (\ref{st}) is satisfied. 
If $c  > c_\#$ then the characteristic equation (\ref{che}) with $p = g'(0)$
has exactly two real solutions $0 < \lambda_2 < \lambda_1, \ \lambda_j = \lambda_j(c).$

\begin{proposition} \label{pr1} Assume  {\rm \bf(H)} and take some $c \geq c^*$. Then  (\ref{1})  has at least one monotone positive traveling front  $u(t,x)= \phi(\nu\cdot x+ct,c)$ propagating at the velocity $c$. Next, for $c < c_\#$ equation (\ref{1}) does not possess
any positive bounded wave solution  $u(t,x)= \psi(\nu\cdot x+ct), \ \psi(-\infty) =0$. 
 Moreover, each such wave solution to (\ref{1})  (if exists) is in fact a monotone front with profile $\psi$ satisfying $\psi'(s) >0$ for all $s \in \R$. Finally, if  $c \not=  c_\#$,
 then   the following asymptotic representation is valid (for an appropriate $s_0, \ j \in \{1,2\}$ and some $\sigma >0$):  
\begin{equation}\label {afe}
(\phi, \phi')(t+s_0,c)= e^{\lambda_j t}(1, \lambda_j) + O(e^{(\lambda_j+ \sigma) t}), \ t \to -\infty. 
\end{equation}
If $c=c_\#$ then besides (\ref{afe}) it may happen that 
\begin{equation}\label {afec}
(\phi, \phi')(t+s_0,c)=-te^{\lambda_j t}(1, \lambda_j) + O(e^{\lambda_j t}), \ t \to -\infty. 
\end{equation}
\end{proposition}
\begin{proof} The existence of fronts for $c \geq c^*$ follows from \cite[Theorem 4]{TAT} while their non-existence for $c < c_\#$ is a well known fact (e.g. see \cite[Theorem 1]{TAT}).    Due to \cite[Corollary 12]{TTT}, the wave profiles $\psi$  are  monotone, with $\psi'(s) >0,$ $ s \in \R$.  The exponential convergence  $\psi(t) \to 0, t \to -\infty,$ is a consequence  of the Diekmann-Kaper theory, see \cite{DK} and \cite[Lemma 3]{AGT}.   Therefore there is $\delta >0$ such that 
$$
g(\psi(t-ch))= \left[g'(0) + r(t)\right]\psi(t-ch),\ \text{where} \ r(t):= \frac{g(\psi(t-ch))}{\psi(t-ch)} -g'(0) = o(e^{\delta t}).
$$
On the other hand, it is easy to see that the convergence  $\psi(t) \to 0, t \to -\infty,$  is not super-exponential, cf.  \cite[Theorem 5.4 and Remark 5.5]{TT}.  Now we can proceed as in  \cite[Remark 5.5]{TT} (where \cite[Proposition 7.2]{FA} should be used) to obtain  asymptotic formulas (\ref{afe}), (\ref{afec}).  
\end{proof}

\subsection{Minimal speed of propagation}
By Proposition \ref{pr1},   after assuming  (\ref{st}), the minimal speed of propagation $c_*$ can be computed  from the characteristic equation (\ref{che}) considered with $p = g'(0)$. Without (\ref{st}), the computation of $c_*$ represents a very difficult task even for non-delayed models \cite{BD,GK,Xin}.  In such a case, the value of $c_*$ depends not only on $g'(0)$ but also on the whole nonlinearity $g$.  
Furthermore,  if $h>0$ and (\ref{st}) does not hold, the situation becomes even more complicated: it is an open question whether  there exists a positive $c_*$ splitting  $\R_+$  on subsets of admissible and non-admissible (semi-) wave speeds.   In the present paper, we answer positively this question at least for $g$ satisfying    assumption (\textbf{H}) with (\ref{gco}) replaced with the slightly more restrictive inequality
\begin{equation}\label{gcos}
\left|g'(u)- g'(0)\right| \leq Cu^\theta, \quad u\in  [0,\delta].
\end{equation}
\begin{theorem} \label{main1}
Suppose that $g$ satisfies  (\ref{gcos})  and {\rm \bf(H)}. Then there exists a positive 
number $c_*$ such that equation (\ref{1}) (a)  for every $c \geq c_*$  possesses  at least one  monotone traveling front $u(t,x)= \phi(\nu\cdot x+ct,c)$; (b)  has not any traveling front propagating at the velocity 
$c < c_*$.  
\end{theorem}
 If $g$ is not monotone, the existence of such $c_*$ remains an unsolved problem.  In any case, for non-monotone $g$,   it is necessary to introduce some adjustments to the definition of traveling front 
solution, replacing it with the concept of semi-wavefront solution, see \cite{TTT,TT}.

\subsection{Uniqueness of wavefronts} More subtle aspects of  uniqueness and stability of wavefronts in (\ref{1}) were studied so far  under the geometric conditions even more restrictive than (\ref{st}). For example, $g''(s) \leq 0$ was required in the main stability theorem of \cite{MeiI}.  Similarly, 
uniqueness (up to a shift)  of each non-critical (i.e.  $c\not= c_*$) monotone traveling front of equation  (\ref{1}) can be deduced from 
\cite[Corollary 4.9]{wlr2} whenever $g$ meets  the conditions: (A1) $g \in C^2[0,\kappa],\ g(x) >0, x \in (0, \kappa)$; (A2)  $g'(\kappa) < 1$ and (\ref{st}) holds; (A3) For every $\delta \in (0, 1)$, there exist $a = a (\delta) > 0,\  \alpha = \alpha (\delta) \geq 0$ 
and $\beta  = \beta(\delta) 	\geq 0 $ 
with $\alpha+\beta > 0$ such that for any $\theta \in (0, \delta]$ and $v \in  [0, \kappa]$,
$$(1 - \theta) g(v) - g ((1 - \theta) v) \leq - a\theta \kappa^\alpha v^\beta.$$
Let us show that  (A3) is stronger than (\ref{st}).  Indeed, after 
dividing the latter inequality by $\theta$ and taking limit as $\theta \to +0$, we find that 
$$-g(v) + g' (v)v  \leq - a \kappa^\alpha v^\beta <0, \quad v \in [0, \kappa].$$
Therefore 
$
g'(v) < g(v)/v, \ v \in (0, \kappa],$
that, after an easy integration, yields 
$$
0 \leq g'(v) < \frac{g(v)}{v} \leq \frac{g(u)}{u} \leq g'(0+) = \lim_{u \to +0}  \frac{g(u)}{u}, \quad v \geq u.
$$
It is clear that the above inequalities are stronger that the Lipshitz condition 
\begin{equation}\label{lc}
|g(u) -g(v)| \leq g'(0)|u-v|, \quad u,v \in [0, \kappa],
\end{equation}
which in turn is more restrictive than (\ref{st}). 
 
Inequality (\ref{lc}) is one of the basic conditions of the uniqueness theory developed by Diekmann and Kaper, cf.  \cite{DK} and \cite{AGT}.   Suppose, for instance, that $g \in C^{1,q}$ in some neighborhood of $0$. Then (\ref{lc}) implies the uniqueness 
of all non-critical \cite{DK} as well as  critical  \cite{AGT}  wavefronts to (\ref{1}). Additionally,  \cite{AGT} establishes
the uniqueness of all fronts propagating at the velocity $c > c_u$ where $c_u$  can be computed (similarly to 
$c_*$ in the sub-tangential case) from the equation 
$$
z^2 -cz - 1 + {\rm ess}\sup_{v \in [0, \kappa]} g'(v) e^{-zch}=0.
$$
An alternative approach to the uniqueness problem is based on  the sliding method developed by Berestycki and Nirenberg  \cite{BN}.  This technique was successfully applied in \cite{chen,co,cdm, mazou} to prove the uniqueness of {\it monotone} wavefronts without  imposing  any Lipshitz condition on $g$.\footnote{In difference, the Diekmann-Kaper theory can be applied to the  non-monotone waves and nonlinearities.} In the present paper, inspired by a recent Coville's work \cite{co}, we use the sliding method to prove the following assertion: 
\begin{theorem}\label{main2} Assume  that  {\rm \bf(H)} is satisfied.  Fix some $c \geq c_*$, and suppose that $u_1(t,x)= \phi(\nu\cdot x+ct)$, $u_2(t,x)= \psi(\nu\cdot x+ct)$ 
are two traveling fronts of equation  (\ref{1}). Then $\phi(s)= \psi(s+s_0), \ s \in \R,$ for some $s_0$. 
\end{theorem}
We note that, when $h >0$, we were not able to drop the condition of strict monotonicity  on $g$ imposed in Theorem \ref{main2} (even while  considering only monotone wavefronts).   If $h=0$, the monotonicity of $g$ is not obligatory. 

The ideas behind the proof of Theorem \ref{main2} combined with  asymptotic description of wavefronts given in \cite{AGT} also allow to derive a new uniqueness result for the following nonlocal lattice system     
\begin{equation} \label{1dnl}
u_{n}'(t) = D [u_{n+1}(t) + u_{n-1}(t)-2u_n(t)] - u_n(t) + \sum_{k \in \Z}\beta (n-k)g(u_{k}(t-h)), \ n \in \Z,
\end{equation}
where $D >0, \beta(k)\geq 0, \ \sum_{k\in \Z}\beta(k)=1$.  Let  $\gamma^\#$ be an extended non-negative real number such that $B(z):= \sum_{k\in \Z}\beta(k)e^{-zk}$ is finite when $z\in[0,\gamma^\#)$ and is infinite when $z > \gamma^\#$.  By Cauchy-Hadamard formula, 
$\gamma^\# = - \limsup_{k \to +\infty} k^{-1}\ln\beta(-k)$, where we adopt the convention that $\ln(0)=-\infty$. Our requirement is  that such 
$\gamma^\#$ is positive   and that $B(\gamma^\#-)=+\infty$. 
\begin{theorem}\label{main4} Assume  {\rm \bf(H)} except for the strict character of the monotonicity of $g$.  Suppose that $w_j(t)= \phi(j+ct)$, $v_j(t)= \psi(j+ct)$ 
are traveling fronts to nonlocal lattice eqution  (\ref{1dnl}) and $c \not=0$. Then there is $s_0$ such that  $\phi(s)= \psi(s+s_0), \ s \in \R.$ 
\end{theorem}
\begin{remark}
 In Theorem \ref{main2}, the inequality $g'(\kappa) < 1$ is formally required. However, our proof uses more weak restriction  $g'(s) \leq 1, \ s \in [\kappa-\sigma,\kappa]$, where $\sigma$ is some positive number. 
\end{remark}
\begin{remark}
In various aspects, Theorem \ref{main4} improves and generalizes on the non-local case the main uniqueness theorem  from  \cite{mazou}.  In difference with the mentioned result, we do not  impose sub-tangency condition (\ref{st}) and we  admit critical (minimal) waves.  Next, the uniqueness result of \cite{mazou} is valid only for profiles having prescribed asymptotic behavior at $-\infty$.   Note also that our proof is rather short and does not  use the monotonicity of profiles. Now, condition $c\not=0$  seems to be essential:  \cite[Proposition 6.7]{cdm} suggests the possibility  of infinitely many wave solutions (perhaps, discontinuous) for $c=0$.
  On the other hand, Theorem \ref{main4} complements the main result of \cite{fwz} (which is valid only for non-critical waves), where (\ref{lc}) was assumed together with the symmetry $\beta(k)=\beta(-k)$.  Even though \cite{fwz} (see also \cite{AGT} for several improvements) allows to consider non-monotone nonlinearity $g$. 
\end{remark}
\subsection{Asymptotic formulas for the wave profiles}

It is well known \cite{ES} that in non-delayed case each critical wavefront which propagates at the velocity $c_* > c_\#$   
(i.e. so called {\it pushed wavefront})  has its profile converging to $0$ more rapidly than the near (i.e. propagating  with the speeds $c \approx c_*$)  non-critical  wavefront  profiles.    This contrasts with the case  $c_*= c_\#$, when the profile of the critical front  (so called {\it pulled wavefront})  converges to $0$ approximately at the same rate as the profile of each near wavefront does.  Similar asymptotics  were also established for wavefront solutions of lattice equation (\ref{1d}) without delay, see \cite[Theorem 3]{chfg}.  Our third main result shows that the pushed fronts to (\ref{1}) obey the same principle: 
\begin{theorem}\label{main3}   Assume    {\rm \bf(H)} and 
suppose that $u(t,x)= \phi(\nu\cdot x+ct)$,  
is a  traveling front to equation  (\ref{1}). Then the following asymptotic represantions are valid
(for an appropriate $s_0$ and some $\sigma >0$):  

1) if $c > c_*$ then $\phi(s+s_0)= e^{\lambda_2 t} + O(e^{(\lambda_2+ \sigma) t}), \ t \to -\infty$, 

2)  if $c = c_* > c_\#$,  then $\phi(s+s_0)= e^{\lambda_1 t} + O(e^{(\lambda_1+ \sigma) t}), \ t \to -\infty$. 
\end{theorem}

The proof of the  second formula is the most difficult part of this theorem. In order to establish that the  pushed fronts to (\ref{1}) satisfy  
2), it suffices to show that each wavefront having asymptotic behavior as in  1) is `robust' with respect to small perturbations of the velocity $c$.  This would imply the existence of wavefronts propagating at the velocity $c' < c_*$ provided that   the critical front behaves as  in 1).  The necessary  perturbation result  is demonstrated here with the use of  upper-lower solutions method. Note that, due to the use of a discontinuous upper solution, application of this  method in the paper is not at all standard.  
%
%
\section{Proof of Theorem \ref{main1}} Fix some $h>  0$ and consider 
$$
{\mathcal{C}}(h) := \{ c \geq 0: \ {\rm equation} \ (\ref{1}) \ {\rm has\ a\ wavefront \ propagating \ at \ the\ velocity\ } c \}.
$$
By Proposition  \ref{pr1},  ${\mathcal{C}}(h) $ contains subinterval $[c^*, +\infty)$ while 
$$
c_*: =\inf {\mathcal{C}}(h) \geq c_\# >0. 
$$
It is easy to see that ${\mathcal{C}}(h)$ is closed (cf. \cite[Lemma 26]{GT}) so that $c_* \in {\mathcal{C}}(h)$. 
Assume  that 
$c_0 \in {\mathcal{C}}(h)$ and let  $c' - c_0 >0$ be small enough to satisfy  $(1+\theta) \lambda_2(c') > \lambda_2(c_0)$.
Observe here that $\lambda_2(c)$ is a decreasing function of $c$.  Let $u(t,x) = \phi(\nu\cdot x +c_0t)$ be a wavefront 
moving at the velocity $c_0$, then $\phi$ solves 
$$
\phi''(t) - c_0\phi'(t) - \phi(t) + g(\phi(t-c_0h)) =0. 
$$
To simplify the notation, we will write  $\lambda_2' := \lambda_2(c'), \ \lambda_2:=\lambda_2(c_0)$. For the convenience of the reader, the proof is divided in several steps.  

\noindent \underline{Step I (Construction of a base function).} Set $\phi_\sigma(t): = \sigma \phi(t)$, where  $\sigma >1$ is close to $1$. We have 
$$
E(t,\sigma):= \phi_\sigma''(t) - c'\phi'_\sigma(t) - \phi_\sigma(t) + g(\phi_\sigma(t-c'h)) \leq  
$$
$$
\phi_\sigma''(t) - c_0\phi'_\sigma(t) - \phi_\sigma(t) + \sigma g(\phi(t-c_0h)) + [(c_0-c')\phi'_\sigma(t)+ g(\phi_\sigma(t-c_0h)) - \sigma g(\phi(t-c_0h))] = 
$$
$$
 (c_0-c')\phi'_\sigma(t)+ g(\phi_\sigma(t-c_0h)) - \sigma g(\phi(t-c_0h)).
$$
By our assumptions, $g(x) = g'(0)x + o(x), \ x \to 0$. Therefore 
$$
g(\phi_\sigma(t-c_0h)) - \sigma g(\phi(t-c_0h))= o (\phi(t-c_0h)), \ t \to - \infty. 
$$
On the other hand, we infer from Proposition \ref{pr1} that 
$$
 (c_0-c')\phi'_\sigma(t)=  (c_0-c')\zeta \sigma\phi(t) (1+ o(1)) = (c_0-c')\zeta \sigma e^{\zeta c_0 h}\phi(t-c_0h) (1+ o(1)). 
$$
for $\zeta \in \{\lambda_1(c_0), \lambda_2(c_0)\}$. 
As a consequence, there exists $T_1$ (which does not depend on $\sigma$) such that, for all $\sigma$ close to 1,  
$$
E(t,\sigma) < 0,\  t \leq T_1. 
$$
Now, due to our assumptions the function 
$
G(u) : = g(u)/u 
$
is $C^1$-smooth within some connected left neighborhood $\mathcal O$ of $\kappa$.  Since $G'(\kappa)= (g'(\kappa)-1)/\kappa <0$ we obtain that 
$$
G(u)-G(v) = G'(\theta)(u-v) < 0, \ u > v, 
$$
for all $u,v$ close to $\kappa$ (say, if $u,v \in {\mathcal I}:= [\kappa - \varepsilon, \kappa] \subset \mathcal O$. Actually, since $\phi(t)$ takes its value in $(0,1)$, without loss of generality, we may assume that $u,v \in {\mathcal I}= [\kappa - \varepsilon, \kappa+\varepsilon]$. Observe  that $\mathcal I$  does not depend on $\sigma$).  As a consequence, if $\phi_\sigma(t-c_0h), \ $ $\phi(t-c_0h) \in \mathcal I$, then 
$$
 (c_0-c')\phi'_\sigma(t)+ g(\phi_\sigma(t-c_0h)) - \sigma g(\phi(t-c_0h)) = 
$$
$$
 (c_0-c')\phi'_\sigma(t)+ \sigma \phi(t-c_0h)(G(\phi_\sigma(t-c_0h)) - G(\phi(t-c_0h))) <0. 
$$
Hence, we have proven that there exists $T_2$ such that,  for all $\sigma$ close to 1,  
$$
E(t,\sigma) < 0,\  t \geq T_2. 
$$
Finally, since uniformly on $[T_1,T_2]$
$$
 \lim_{\sigma \to 1}(c_0-c')\phi'_\sigma(t)+ g(\phi_\sigma(t-c_0h)) - \sigma g(\phi(t-c_0h)) = (c_0-c')\phi'(t) <0,
$$
we conclude that $E(t,\sigma) <0$ for all $t \in \R$ and $\sigma$ close to 1. 

\vspace{2mm}

\noindent  \underline{Step II (Construction of an upper solution).}  

For $a:=b^2,\  b \in (0,1]$, set $\phi_{b}(t): = \phi_\sigma(t) + ae^{\lambda'_2 t}+ be^{\lambda_2 t}$, where $\lambda'_2 =\lambda_2(c'),$ $ \ \lambda_2 =\lambda_2(c_0)$. Let $T_3= T_3(b)$ be that unique point where 
$\phi_b(T_3(b))=\kappa$. It is clear that $\phi'_b(T_3) >0$.  Next, we find  that 
$$
E_+(t,b):= \phi_b''(t) - c'\phi'_b(t) - \phi_b(t) + g(\phi_b(t-c'h)) =   E(t,\sigma)+ b\chi(c',\lambda_2)e^{\lambda_2 t} +
$$
$$
g( \phi_\sigma(t-c'h) + ae^{\lambda'_2 (t-c'h)}+be^{\lambda_2(t-c'h)})-   g(\phi_\sigma(t-c'h)) - g'(0)(ae^{\lambda'_2 (t-c'h)}+be^{\lambda_2(t-c'h)})\leq 
$$
$$
 E(t,\sigma) +b\chi(c',\lambda_2)e^{\lambda_2 t}+$$
$$
C(ae^{\lambda'_2 (t-c'h)}+be^{\lambda_2(t-c'h)})( \phi_\sigma(t-c'h) + ae^{\lambda'_2 (t-c'h)}+
be^{\lambda_2(t-c'h)})^\theta\leq  E(t,\sigma) +
$$
$$
be^{\lambda_2 t}\left(\chi(c',\lambda_2)+3Ce^{-\lambda_2 c'h}(be^{(\lambda'_2-\lambda_2) (t-c'h)}+1)
(\phi^\theta_\sigma(t-c'h)+2b^\theta e^{\lambda'_2\theta (t-c'h)} )\right)\leq E(t,\sigma) +
$$
$$
be^{\lambda_2 t}\left(\chi(c',\lambda_2)+C_1 
\phi^\theta_\sigma(t-c'h)+C_2b^\theta e^{(\lambda'_2(1+\theta)-\lambda_2) (t-c'h)} +C_3b \phi^\theta_\sigma(t-c'h)e^{(\lambda'_2-\lambda_2)t }\right)\leq 
$$
$$
E(t,\sigma) + be^{\lambda_2 t}\left(\chi(c',\lambda_2)+C_4 e^{\nu t}\right), \quad t \leq T_4,
$$
for some positive $\nu, \ C_j$ and negative $T_4$ (which does not depend on $b$).   Since $\chi(c',\lambda_2)<0$, we may choose  $T_4$ is  such a way that 
$E_+(t,b) <0$ for all $t \leq T_4,$ $ b \in (0,1]$.  On the other hand, we know that, uniformly on each compact interval,  $E_+(t,b) \to E(t,\sigma),  b \to 0+$.  Therefore $E_+(t,b) <0$ for all $t \leq T_3(0+)+1$ for all sufficiently small $b$. 

\noindent Now, let us  define an upper solution $\phi_+$ by 
$
\phi_+(t): = \min\{\kappa, \phi_b(t)\}. 
$
It is clear that  $\phi_+(t)$ is continuous and piece-wise $C^1$  on $\R$, being $t_0:= T_3(b)$  the unique point  of discontinuity 
of the derivative where $\Delta \phi_+'|_{t_0}:= \phi_+'(t_0+)- \phi_+'(t_0-) = - \phi_b'(t_0-) <0$. 

\noindent \underline{Step III (Construction of a lower solution).} 
Consider the  following concave monotone  linear rational  function 
$$p(x):= \frac{g'(0)x}{1 +Ax} \leq g'(0)x, \ x \geq 0,  \ A: =  2\frac{g'(0)-1}{\kappa}, \ p(0)=0, \ p(\frac{\kappa}{2})= \frac{\kappa}{2},$$
and set $g_-(x) : = \min \{g(x),p(x)\}$. It is clear that $g_-$ is continuous and increasing and that 
$$g_-'(0)= g'(0), \ g_-(0)=0, \ g_-(\frac{\kappa}{2})= \frac{\kappa}{2}, \quad g_-(x) \leq g'(0)x, \ x \geq 0.$$
Moreover, in some right neighborhood of $0$, function $g_-(x)$ meets the smoothness condition of {\rm \bf(H)}. 
This implies the existence of a monotone positive function $\phi_-, \ \phi_-(-\infty)=0,$ $\phi_-(+\infty)= \kappa/2,$ satisfying the equation
$$
\phi_-''(t) - c'\phi_-'(t) - \phi_-(t) + g_-(\phi_-(t-c'h)) =0, 
$$
e.g., see \cite[Theorem 4]{TAT}. Due to the  property $g_-(x) \leq g'(0)x, \ x \geq 0$, we also know that 
$$
(\phi_-, \phi'_-)(s+s_0,c)= e^{\lambda'_2 t}(1, \lambda'_2) + O(e^{(\lambda'_2+ \sigma) t}), \ t \to -\infty. 
$$
Finally,  since $g_-(x) \leq g(x)$ we obtain that
$$
\phi_-''(t) - c'\phi_-'(t) - \phi_-(t) + g(\phi_-(t-c'h)) \geq 0. 
$$
\underline{Step IV (Iterations).} 
Comparing asymptotic representations of monotone functions  $\phi_-(t)$ and $\phi_+(t)$  at $-\infty$, we   find easily that 
$$
\phi_-(t+s_1) \leq  \phi_+(t), \quad t \in \R, 
$$ 
for some appropriate $s_1$. Simplifying, we will suppose that $s_1=0$.  
In the next stage of the proof, we need the following simple result:
\begin{lemma}\label{imp} Let $\psi: {\mathbf R} \to {\mathbf R}$ be a bounded classical solution of
the second order  impulsive equation
\begin{equation}\label{imin}
\psi'' -c\psi' - \psi =f(t), \quad \Delta \psi|_{t_j} = \alpha_j,
\quad \Delta \psi'|_{t_j} = \beta_j,
\end{equation}
where $\{t_j\}$ is a finite increasing sequence, $f: {\mathbf R} \to {\mathbf R}$ is
bounded and continuous at every $t \not= t_j$ and the operator
$\Delta$ is defined by $\Delta w|_{t_j} := w(t_j+)-w(t_j-)$. Assume
that equation $z^2 -c z -1 =0$ has two real roots $\xi_1< 0<\xi_2, \ \xi_j= \xi_j(c)$.
Then
\begin{eqnarray}\label{pfon}&& \psi(t)=\frac{1}{\xi_1
- \xi_2} \left(\int^t_{-\infty} e^{\xi_1 (t-s)}f(s)ds +
\int_t^{+\infty}e^{\xi_2 (t-s)}f(s)ds\right) \\ &&  \nonumber
 +\frac{1}{\xi_2 - \xi_1}\left[\sum_{t<t_j}e^{\xi_2(t-t_j)}(\xi_1\alpha_j
 -\beta_j)+\sum_{t>t_j}e^{\xi_1(t-t_j)}(\xi_2\alpha_j-\beta_j)\right], \
 \ t\not=t_j.
\end{eqnarray}
\end{lemma}
\begin{proof}
See \cite{TPT}.  Alternatively,  it can be checked by a
straightforward substitution that $\psi$ defined by (\ref{pfon})
verifies  equation (\ref{imin}). \end{proof}
Similarly to \cite{wz}, we also consider the monotone integral operator 
$$
(\mathcal{A}\phi)(t): = \frac{1}{\xi'_2
- \xi'_1} \left(\int^t_{-\infty} e^{\xi'_1 (t-s)}g(\phi(s-c'h))ds +
\int_t^{+\infty}e^{\xi'_2 (t-s)}g(\phi(s-c'h))ds\right), $$
where $\xi'_j: = \xi_j(c').$ Using properties of  functions $\phi_-(t)$ and $\phi_+(t)$,  we deduce from 
Lemma \ref{imp} that 
$$
\phi_-(t) \leq (\mathcal{A}\phi_-)(t) \leq (\mathcal{A}^2\phi_-)(t)\leq \dots \leq (\mathcal{A}^2\phi_+)(t) \leq  (\mathcal{A}\phi_+)(t) \leq    \phi_+(t), \quad t \in \R. 
$$
The latter  implies (see \cite{wz} for more detail) the existence of a monotone solution $\phi(t)$ such that 
$$
(\mathcal{A}\phi)(t) = \phi(t),  \quad \phi_-(t) \leq  \phi(t) \leq  \phi_+(t), \quad t \in \R. 
$$
This amounts to the existence of a wavefront propagating at velocity $c'$. Moreover, the latter estimations shows that, for some $s_0$ and positive $\delta$, 
\begin{equation}\label{ff}
\phi(s+s_0)= e^{\lambda'_2 t} + O(e^{(\lambda'_2+ \delta) t}), \ t \to -\infty.
\end{equation}
Finally, to prove that ${\mathcal{C}}(h)$ coincides with the interval $[c_*, \infty)$, let us consider the open set 
$O= [c_*,\infty)\setminus {\mathcal{C}}(h)$.  If $O \not=\emptyset$, we  take one connected component 
of $O$, say $(c_0,c_1)$.  Since $c_0 \in {\mathcal{C}}(h)$, there is some  $c'\in (c_0,c_1)$ such that 
$c' \in {\mathcal{C}}(h)$, in contradiction to the definition of $O$. Therefore ${\mathcal{C}}(h)=[c_*, \infty)$.

\section{Proof of Theorem \ref{main2}} In our proof which was inspired by Coville work \cite{co},  we invoke the sliding method developed by Berestycki and Nirenberg  \cite{BN,chen,co,cdm}. 
\begin{lemma} \label{l6} Fix some $c \geq c_*$ and suppose that $\phi, \psi$ are two wavefront profiles such that, for some finite $T$, 
\begin{equation}\label{nado}
\phi(t) < \psi(t),  \quad t <  T. 
\end{equation}
Then $\phi(t) < \psi(t)$  for all $t \in \R$. 
\end{lemma}
\begin{proof} 
Set $a_* = \inf\mathcal{A}$ where
$$
\mathcal{A}: = \{a \geq 0: \psi(t)+a \geq \phi(t), \ t \in \R \}. 
$$
Note  that $\mathcal{A}\not=\emptyset$ since $[\kappa, +\infty) \subset \mathcal{A}$. Moreover, $a_* \in \mathcal{A}$.

Now, if $a_* =0$ then  $\psi(t) \geq \phi(t), \ t \in \R.$ We claim that, in fact, $\psi(t) > \phi(t), \ t \in \R.$
Indeed, otherwise we can suppose that $T$ is such that $\phi(T)=\psi(T)$.  In this way, the difference
$\psi(t)- \phi(t) \geq 0$ reaches its minimal value $0$ at $T$, while $\psi(T-ch) > \phi(T-ch)$. 
But then we get  a contradiction: 
\begin{eqnarray}\label{pt}
0= (\psi''(T)-\phi''(T)) - c(\psi'(T)-\phi'(T)) - (\psi(T)-\phi(T)) +\nonumber \\
(g(\psi(T-ch))- g(\phi(T-ch))) >0. 
\end{eqnarray}
In this way,  Lemma \ref{l6} is proved when $a_* =0$ and consequently we may assume that $a_* >0$. 
Let $\sigma >0$ be small enough to satisfy
$$\max_{s \in [\kappa-\sigma, \kappa]} g'(s) \leq 1.$$
\underline{Case I.} 
First, we assume that $T$ is such that, additionaly
\begin{equation}\label{blun}
\phi(t), \psi(t) \in (\kappa-\sigma, \kappa), \ t \geq T-ch.
\end{equation} 
In such a case non-negative function 
$$w(t): = \psi(t) +a_* - \phi(t), \quad w(\pm \infty) = a_*>0, $$  
reaches its minimal value $0$  at some leftmost point $t_m$, where 
$$
\psi(t_m) - \phi(t_m) =-a_*, \quad \psi'(t_m) - \phi'(t_m) =0, \quad \psi''(t_m) - \phi''(t_m) \geq 0. 
$$
Since $\psi(t_m) < \phi(t_m)$, we have that $t_m > T$, so that 
$$\psi(t_m-ch) , \phi(t_m-ch) \in (\kappa-\sigma, \kappa).$$
In consequence,  for some $\theta \in (\kappa-\sigma, \kappa)$, 
\begin{eqnarray}\label{pt2}
0= (\psi''(t_m)-\phi''(t_m)) - c(\psi'(t_m)-\phi'(t_m)) - (\psi(t_m)-\phi(t_m)) + \nonumber \\
(g(\psi(t_m-ch))- g(\phi(t_m-ch))) \geq a_* + g(\psi(t_m-ch))- g(\phi(t_m-ch)) \geq 
\end{eqnarray}
$$
\left\{
\begin{array}{ll}   a_* >0,
    &\ {\rm if} \   \psi(t_m-ch) \geq  \phi(t_m-ch);\\
   a_* + g'(\theta)(\psi(t_m-ch)- \phi(t_m-ch))>  0,  & \ {\rm if} \   \psi(t_m-ch) - \phi(t_m-ch) \in [-a_*,0). 
\end{array}%
\right.
$$
a contradiction. Observe that the strict inequality in the last line 
can be explained in the following way.  
The sign "$\geq$" can be replaced with "$=$" in 
$$a_* + g(\psi(t_m-ch))- g(\phi(t_m-ch)) = a_* + g'(\theta)(\psi(t_m-ch)- \phi(t_m-ch)) \geq 0, $$ 
 if and only if  $g'(\theta)=1$ and $\psi(t_m-ch)- \phi(t_m-ch)= -a_*$.  This, however, is impossible due to the 
 definition of $t_m$ as the leftmost point where $w(t_m)=0$. 
\noindent \underline{Case II.} If (\ref{blun}) does not hold,  then, due to the convergence of profiles at $+\infty$,  we can find large $\tau>0$ and $T_1 > T$  such that 
$$
\psi(t+\tau) > \phi(t),\  t < T_1, \quad \phi(t),\psi(t+\tau) \in (\kappa-\sigma, \kappa), \ t \geq T_1-ch.  
$$
Therefore, in view of  the result established in Case I, we obtain that 
 \begin{equation}\label{bla}
\psi(t+\tau) > \phi(t), \quad t \in \R.   
\end{equation} 
Define now $\tau_*$ by
$$
\tau_*:= \inf \{\tau \geq 0: \ {\rm inequality} \ (\ref{bla}) \ {\rm holds}\}. 
$$
It is clear that $\psi(t+\tau_*) \geq \phi(t), \ t \in \R$.  Since, in addition, 
$$\psi(t+\tau_*) \geq \psi(t) > \phi(t), \ t <T,$$ we conclude that $\psi(t+\tau_*) > \phi(t), \ t \in \R$, cf. (\ref{pt}). 
Now, if  $\tau_*=0$, then Lemma \ref{l6} is  proved.  
Otherwise, $\tau_* >0$ and for each $\varepsilon \in (0,\tau_*)$ there exists a unique $T_\varepsilon > T$ such that 
$$
\psi(t+\tau_*-\varepsilon) > \phi(t), \ t <T_\varepsilon, \  \psi(T_\varepsilon+\tau_*-\varepsilon) = \phi(T_\varepsilon). 
$$
It is immediate to see that $\lim T_\varepsilon = +\infty$ as $\epsilon \to 0+$.  Indeed, if $T_{\varepsilon_j} \to T'$  for some finite $T'$ and $\varepsilon_j \to 0+$, then we get a contradiction: $\psi(T'+\tau_*) = \phi(T')$.  Therefore, if $\varepsilon$ is small, then 
$$
\psi(t+\tau_*-\varepsilon), \phi(t) \in (\kappa-\sigma, \kappa), \ t \geq T_\varepsilon-ch, 
$$
that is  $\psi(t+\tau_*-\varepsilon) $ and $\phi(t)$ satisfy condition (\ref{blun}) required in  Case I.  Thus we get 
$\psi(t+\tau_*-\varepsilon) > \phi(t)$ for all $t \in \R$, a contradiction to the definition of $h_*$.  This means that 
$\tau_* =0$ and the proof of Lemma \ref{l6} is  completed.   
\end{proof}
\begin{corollary} \label{coco}For a fixed $c \geq c_*$, both $\phi$ and $\psi$ have the same type of asymptotic 
behaviour at $-\infty$ described in Proposition \ref{pr1}. 
\end{corollary}
\begin{proof} For example, suppose that $\phi(t) \sim e^{\lambda_2 t}$ and $\psi(t) \sim e^{\lambda_1 t}$ as $t \to -\infty$. 
Then for every fixed $\tau \in \R$ there exists $T(\tau)$ such that $\psi(t+\tau) > \phi(t)$ for all $t < T(\tau)$.  Applying 
Lemma \ref{l6}, we obtain that  
 $\psi(s) > \phi(t)$ for every $s:=t+\tau, t  \in \R,$ what is clearly false. 
\end {proof}
\begin{proof}[Proof of Theorem \ref{main2}] By Corollary \ref{coco}, we can suppose that $\psi(t)$ and $\phi(t)$ have the same type (described in Proposition \ref{pr1}) of asymptotic behavior at $-\infty$. 
Consequently,  $\psi(t+\tau), \phi(t)$ satisfy condition (\ref{nado}) of  Lemma \ref{l6} for every small $\tau >0$. But then 
 $\psi(t+\tau) >  \phi(t)$ for every small $\tau >0$ that yields   $\psi(t) \geq  \phi(t), \ t \in \R$. By symmetry,  we also find that 
$\phi(t) \geq  \psi(t), \ t \in \R$, and Theorem \ref{main2} is proved. 
\end {proof}
\section{Proof of Theorem \ref{main4}}  It is easy to see that each wave profile $\varphi$  verifies
\begin{equation}\label{sy1}
c\varphi'(t)= D[\varphi(t+1)+\varphi(t-1) -2\varphi (t)] - \varphi (t) + \sum_{k\in \Z}\beta(k)g(\varphi(t-k-ch)).
\end{equation}
First we note that $\varphi(t)$ takes its value in $(0,\kappa)$. Indeed,  suppose for a moment that $s_0$ is the leftmost point where $M: = \varphi(s_0) = \sup_{s \in \R} \varphi (s) \geq \kappa$.  Then  $\varphi'(s_0)=0$ and $\varphi(s_0+1)+\varphi(s_0-1) -2\varphi(s_0) <0, \ g(\varphi(s_0-k-ch)) \leq g(M)$. Consequently, $M< g(M), \ M \geq \kappa$, a contradiction.  

Second, we claim that  $\varphi(t)$ is strictly increasing at $-\infty$ (we believe that $\varphi$ is monotone on $\R$, cf. \cite{mazou}: however, for our purpose it suffices to establish the monotonicity of $\varphi(t)$  on some of intervals  $(-\infty, \rho)$).   Consider the characteristic function
$$
\tilde\chi(z,c):= 1+2D+cz -D(e^z +e^{-z})-g'(0)e^{-chz}\sum_{k\in \Z}\beta(k)e^{-kz}
$$ 
and the bilateral Laplace transform $\Phi(z):= \int_{\R}e^{-zs}\varphi(s)$. For each fixed  $c\not=0$ function 
$\tilde\chi(z,c)$ is analytic in the region $\Pi_1 = \{0 < \Re z < \gamma_\#\}$ of the complex plane $\C$ and has 
a finite number of roots in any subregion $\{0< \epsilon < \Re z <  \gamma_\#-\epsilon\}$, see \cite[Lemma 3.1]{fwz}.   Next, it was proved in \cite{AGT} that, under the conditions of Theorem \ref{main4},  $\Phi(z)$ is analytic in some maximal vertical strip $\Pi= \{0 < \Re z < \lambda \}\subset \Pi_1$ where $\lambda < \gamma_\#$ is a positive 
root (in difference with \cite{fwz}, not necessarily minimal and simple) of the equation $\tilde\chi(z,c)=0$.  Again using \cite[Lemma 3.1]{fwz} (or, alternatively, \cite[Lemma 2]{AGT}),  we obtain that there exists $r >0$ such that
\begin{equation}\label{1zl}
\{\lambda\} = \{z \in \C: \tilde \chi(z,c) =0, \ \lambda -r < \Re z < \lambda+r\}.   
\end{equation}
Moreover, $\varphi(t) = O(e^{\gamma t}), \ t \to -\infty,$ for each $\gamma \in (0,\lambda)$.  See Corollaries 1,3 and Theorem 6 in \cite{AGT} for more detail.  Yet we will need a stronger result: 
\begin{lemma} \label{prL} Under assumptions of Theorem \ref{main4},  we have that  
$$
\varphi (s+s_0,c)= (-t)^je^{\lambda t} + O(e^{(\lambda+ \sigma) t}),  \  \varphi'(t)= 
\lambda \varphi (t) (1+ o(1)), \ t \to -\infty.
$$
for an appropriate $s_0, \ j \in \{0,1\}$ and some $\sigma >0$. As a consequence, $\varphi$ is strictly increasing on some maximal open interval $(-\infty, \rho)$. 
\end{lemma}
\begin{proof} Here, we follow the proof of Theorem 3 (Step I) in \cite{AGT}. Set 
$$
\mathcal{D}(t):= \sum_{k\in \Z}\beta(k)( g'(0)\varphi(t-k-ch)-g(\varphi(t-k-ch))), 
$$
Take $C,\delta, \theta$ as in  {\bf {(H)}}.    Observe that without restricting the generality, we can assume that $(1+\theta)\lambda < \gamma_\#$. 
Since equation (\ref{sy1})  is translation invariant, we can suppose that $\varphi(t) < \delta$ for $t \leq 0$. 
Applying the bilateral Laplace transform to (\ref{sy1}), we obtain that 
$$
\tilde \chi(z,c)\Phi(z) = \int_{\R}e^{-zt}\mathcal{D}(t)dt =:\mathbf{D}(z), \ z \in \Pi. 
$$
We claim that, in fact,  function $\mathbf{D}$ is analytic  in the region  $\Pi_\alpha=\{z:\Re z  \in (0, (1+\theta)\lambda)\}$.  Indeed, we have
$$
\mathbf{D}(x+iy) = \int_{\R}e^{-iyt}[e^{-xt}\mathcal{D}(t)]dt. 
$$
Given $x: = \Re z  \in (0, (1+\theta)\lambda)$, we choose  $x'$ sufficiently close from the left to $\lambda$ to satisfy
$
-x+(1+\theta)x' >0.
$
Then $\varphi(t) \leq  C_x e^{ x't}, \ t \in \R,$  for some positive $C_x$  and  
$$
|\mathcal{D}(t)|\leq 
 C\sum_{k\geq t-ch}\beta(k)|\varphi(t-k-ch)|^{1+\theta}  +  
\kappa(1+g'(0))\sum_{k < t-ch}\beta(k) \leq  
$$
$$
e^{(1+\theta)x't}C_1\sum_{k\geq t-ch}\beta(k)e^{-x'(1+\theta)(k+ch)} +  
\kappa(1+g'(0))\sum_{k < t-ch}\beta(k) e^{-x'(1+\theta)(k+ch -t)}
\leq 
$$
$$
e^{(1+\theta)x't}\left[ C_2+  \kappa(1+g'(0))\right]\sum_{k \in \Z}\beta(k) e^{-x'(1+\theta)k} \leq C_*e^{(1+\theta)x't}, \ t \in \R. 
$$

Since clearly $\mathcal{D}(t)$ is bounded on $\R$, we find  that 
$e^{-xt}\mathcal{D}(t)$ belongs to $L^k(\R)$, for each $k \in [1, \infty]$ and $x \in (0, (1+\theta)\lambda)$.  In consequence, $\mathbf{D}$ is analytic  in $\Pi_\alpha$. In addition, for each $x \in (0, (1+\theta)\lambda)$ the function 
$\mathbf{d}_x(y):=\mathbf{D}(x+i  y )$ is bounded and square integrable on $\R$.  Also, for each 
vertical line $L_x:= \{x+it, \ t \in \R\}$ where $\tilde \chi (x+it) \not=0$, we have that $\tilde \chi (x+it) \sim cit, \ |t| \to \infty$.  Thus 
$1/\tilde \chi (x+it)$ is  square integrable on $\R$ as well. Consequently, for each fixed $x \in (0, (1+\theta)\lambda)$ such that $L_x$ does not contain zeros of $\tilde\chi(z)$, function $\mathbf{D}(x+i y )/\tilde \chi (x+iy)$ is integrable on $\R$ . 

As we have mentioned, $\tilde\chi(z,c)$ is  analytic in the domain $\Pi_\alpha$,
while 
$
\Phi(z) = {\mathbf{D}(z)}/\tilde \chi (z,c)
$
is analytic in  $\Re z  \in (0, \lambda)$ and meromorphic in $\Pi_\alpha$. In virtue of  (\ref{1zl}), we can suppose that 
$\Phi(z)$ has a unique singular point $\lambda$ in $\Pi_\alpha$ which is either simple or double pole. 

Now, for some $x'' \in (0,\lambda)$,  using the inversion theorem for the Laplace transform, we obtain that 
$$
\varphi(t) = \frac{1}{2\pi i}\lim_{N \to +\infty} \int_{x''-iN}^{x''+iN}\frac{e^{zt}\mathbf{D}(z)}{\tilde\chi (z,c)}dz, \ t \in \R. 
$$
If $x \in (\lambda, (1+\theta)\lambda)$ then 
$$
\int_{x''-iN}^{x''+iN}\frac{e^{zt}\mathbf{D}(z)dz}{\tilde\chi (z,c)} = \left( 
 \int_{x-iN}^{x+iN}+\int_{x''-iN}^{x-iN}- \int_{x''+iN}^{x+iN}\right)\frac{e^{zt}\mathbf{D}(z)dz}{\tilde\chi (z,c)} - 2\pi i{\rm Res}_{z=\lambda}\frac{e^{zt}\mathbf{D}(z)}{\tilde\chi (z,c)}.
$$
Since, by \cite[Corollary 2]{AGT}, 
$$
\lim_{N \to +\infty}\max_{z \in [x''\pm iN, x\pm iN]}(|\mathbf{D}(z)|+|1/\tilde\chi (z,c)|)=0,
$$
we  conclude  that, for each fixed $t \in \R$ 
$$
\lim_{N \to +\infty}\int_{x''\pm iN}^{x\pm iN}\frac{e^{zt}\mathbf{D}(z)}{\tilde\chi (z,c)}dz=0.
$$
Observe also that   function $\tilde \chi(z,c)$ does not have zero other than $\lambda$  in a small strip centered at $\Re z = \lambda$. 
Therefore
$$
\varphi(t) =  - {\rm Res}_{z=\lambda}\frac{e^{zt}\mathbf{D}(z)}{\tilde\chi (z,c)} + \frac{e^{xt}}{2\pi}\int_{\R}\frac{e^{iyt}\mathbf{d}_x(y)}{\tilde\chi(x+iy,c)}dy.
$$
Since 
$$
{\rm Res}_{z=\lambda}\frac{e^{zt}\mathbf{D}(z)}{\tilde\chi (z,c)}  = \frac{e^{\lambda t}\mathbf{D}(\lambda)}{\tilde\chi' (\lambda,c)},  \quad {\rm if} \ \chi' (\lambda,c) \not=0 , 
$$
$$
{\rm Res}_{z=\lambda}\frac{e^{zt}\mathbf{D}(z)}{\tilde\chi (z,c)}  = \frac{2e^{\lambda t}}{\tilde\chi'' (\lambda,c)}  \left(t \mathbf{D}(\lambda) + \mathbf{D}'(\lambda)- \mathbf{D}(\lambda)\frac{\tilde\chi'''(\lambda,c)}{3\tilde\chi''(\lambda,c)}\right),  \quad {\rm if} \  \chi' (\lambda,c) =0,
$$
we get the desired representation.  It should be noted here that $\tilde\chi''(\lambda,c) <0$, that 
$$
\lim_{|t| \to +\infty}\int_{\R}\frac{e^{iyt}\mathbf{d}_x(y)}{\tilde\chi(x+iy,c)}dy =0, \quad {\rm Res}_{z=\lambda}\frac{e^{zt}\mathbf{D}(z)}{\tilde\chi (z,c)}\not=0.
$$
Indeed, if the latter  residue were equal to $0$, then $\Phi(z)$ would not have a pole at $\lambda$. 

Finally, it is easy to check that 
$
c\varphi'(t)= 
D[\varphi(t+1)+\varphi(t-1) -2\varphi (t)] - \varphi (t) $
$+ \sum_{k\in \Z}\beta(k)g'(0)\varphi(t-k-ch) + \mathcal{D}(t)=
c\lambda \varphi (t) (1+ o(1)), \ t \to -\infty.
$
\end{proof}

Next, we claim that the statement of  Lemma \ref{l6} is also valid for solutions of (\ref{sy1}).  Regardless the fact that we do not know either  wavefronts are monotone on whole real line or they are  not, the proof of Case II can be repeated almost literally.  The monotonicity of wavefronts on $(-\infty, \rho)$ will be sufficient for this purpose. For instance, let us prove the 
following 

\begin{lemma} Under the assumptions of Lemma  \ref{l6}, there are large $\tau>0$ and $T_1 > T$  such that 
$$
\psi(t+\tau) > \phi(t),\  t < T_1, \quad \phi(t),\psi(t+\tau) \in (\kappa-\sigma, \kappa), \ t \geq T_1-ch.  
$$
\end{lemma}
\begin{proof} Due to the monotonicity of $\phi$ and $\psi$ at $-\infty$, we find that for every $\tau \geq 0$
there exists $T(\tau)$ such that 
$$
\psi(t+\tau) > \phi(t),\  t < T(\tau), \quad \phi(T(\tau))=\psi(T(\tau)+\tau).   
$$
Let us prove that $T(\tau)$ is bounded from below on $\R_+$. Indeed, otherwise there exists a converging sequence 
$\tau_j$ such that $T(\tau_j) \to - \infty$. In turn, this  forces $T(\tau_j) + \tau_j \to -\infty$.  But then we  may 
use the monotonicity properties of $\phi, \psi$ in order to get a contradiction:  
$$
\phi(T(\tau))=\psi(T(\tau)+\tau) > \psi(T(\tau)).  
$$
Since $\phi(s) < \kappa, \ s \in \R$, we  deduce in a similar way  that the sequence $\{T(\tau_j)\}$ can not have a finite limit as $\tau_j \to +\infty$.  Thus $T(\tau) \to +\infty$ as $\tau \to +\infty$.  Since $\phi(+\infty) = \psi(+\infty) = \kappa$, the remainder of the proof is straightforward. 
\end{proof}

\noindent Now, the following main changes should be  introduced in  the proof of  Lemma \ref{l6}:  
\vspace{2mm}

1. Set $\Delta (t) = \psi(t)-\phi(t)$. Instead of (\ref{pt}), we then have that $\Delta(T)= \Delta'(T)=0$,  
 \begin{eqnarray}
0= D[\Delta(T+1)+\Delta(T-1) -2\Delta (T)] - \Delta(T) +  \nonumber \\
\sum_{k\in \Z}\beta(k)\left(g(\psi(T-k-ch))- g(\phi(T-k-ch))\right) > 0. \nonumber
\end{eqnarray}
Here (non-strict) monotonicity of $g$ is sufficient because of 
$$\Delta(T+1)+\Delta(T-1) -2\Delta (T) \geq \Delta(T-1) >0, \ g(\psi(s)) \geq g(\phi(s)), \ s \in \R. $$

2. If $a_* >0$, we take small positive $\sigma >0$ and integer $N_1>0$ such that 
$$\kappa \sum_{|k| \geq N_1}\beta(k) \leq 0.5a_* (1- \max_{s\in [\kappa-\sigma,\kappa]}g'(s))$$ 
and then we assume additionally that $T$ is such that
$$
\phi(t), \psi(t) \in (\kappa-\sigma, \kappa), \ t \geq T-N_1-ch.
$$
3. Similarly, in (\ref{pt2}),  the expression $g(\psi(t_m-ch))- g(\phi(t_m-ch))$ should be replaced with 
$$
\sum_{k\in \Z}\beta(k)\left(g(\psi(t_m-k-ch))- g(\phi(t_m-k-ch))\right) \geq -0.5 a_*(1- \max_{s\in [\kappa-\sigma,\kappa]}g'(s))+ 
$$
$$ \sum_{|k| < N_1}\beta(k)g'(\theta_k)(\psi(t_m-k-ch) - \phi(t_m-k-ch)) \geq -0.5 a_*(1+ \max_{s\in [\kappa-\sigma,\kappa]}g'(s)).
$$
As a result, we get again a contradiction: 
\begin{eqnarray*}\label{pt22}
0= D[\Delta(t_m+1)+\Delta(t_m-1) -2\Delta(t_m)]  - \Delta(t_m)+ \nonumber \\
\sum_{k\in \Z}\beta(k)\left(g(\psi(t_m-k-ch))- g(\phi(t_m-k-ch))\right) > 0.5 a_*(1- \max_{s\in [\kappa-\sigma,\kappa]}g'(s)) \geq 0.
\end{eqnarray*}

To finalize the proof of Theorem \ref{main4},  it suffices to repeat the last two paragraphs of the third section. 

\section{Proof of Theorem \ref{main3}} In virtue of the front uniqueness, the first statement of Theorem \ref{main3} was already proved in the previous section (cf. (\ref{ff})) so we have to consider the case $c= c_*$ only. 
Suppose, contrary to our claim,  that
$$\phi(t+s_0)= e^{\lambda_2 t} + O(e^{(\lambda_2+ \delta) t}), \ t \to -\infty, \ \lambda_j:= \lambda_j(c_*), $$
(without restricting the generality, we can assume that $s_0 =0$), take some $c' < c_*$ close to $c_*$ and consider the following piecewise continuous function
$$
\phi_+(t):= \left\{%
\begin{array}{lll}
Me^{\rho t }+ a e^{\lambda_2't },
    & \hbox{when}\  t \leq T_1,
   \\ \phi(t)+\epsilon,
    & \hbox{when} \ t \in(T_1, T_2],   \\
      \kappa, &  \hbox{when}\
      T> T_2,
\end{array}%
\right.\nonumber $$
where $\lambda_2': = \lambda_2(c') > \lambda_2, \ \rho =\lambda_2(1+\theta) > \lambda_2', \ M, a, \epsilon  >0,\  a\ll \epsilon \ll 1, \ M \gg 1,$ and 
$$
Me^{\rho T_1}+ a e^{\lambda_2'T_1} = \phi(T_1), \quad  \phi(T_2)+ \epsilon = \kappa. 
$$
For sufficiently large $M$ and small  $\epsilon >0$, the above definitions yield  large negative $T_1= T_1(M,a)$ and large positive $T_2(\epsilon)$.  Therefore, if $M$ is sufficiently large and $a, c_*-c'>0$ are sufficiently small, then we can suppose that,  for all $t \leq T_1$,
$$
E_+:= E_+(t,\epsilon,a,M,c'):= \phi_+''(t) - c'\phi'_+(t) - \phi_+(t) + g(\phi_+(t-c'h)) =  
$$
$$
Me^{\rho t} \left (\chi(\rho,c') + \left[\frac{g(\phi_+(t-c'h))}{\phi_+(t-c'h)} - g'(0)\right]e^{-\rho c'h}\right) +
$$
$$
ae^{\lambda_2' t} \left[\frac{g(\phi_+(t-c'h))}{\phi_+(t-c'h)} - g'(0)\right]e^{-\lambda_2' c'h} <
0.5 Me^{\rho t} \chi(\rho,c') +
C_1ae^{\lambda_2' t} \left[Me^{\rho t }+ a e^{\lambda_2't }\right]^\theta
$$
$$
 \leq 0.5 Me^{\rho t} \chi(\rho,c') +
C_1ae^{\lambda_2' (1+\theta)t} \left[Me^{(\rho-\lambda_2') t }+ a\right]^\theta \leq 
$$
$$
0.5 Me^{\rho t} \chi(\rho,c') +
C_1aMe^{\lambda_2' (1+\theta)t}  <Me^{\rho t} (0.5 \chi(\rho,c') +
C_1a) <0. 
$$

Moreover,  since $\rho > \lambda_2$,  we also can choose $T_1\gg a$ in such a way that 
$$
\phi'_+(T_1-) \approx \rho \phi_+(T_1) >  \phi'_+(T_1+) \approx \lambda_2   \phi_+(T_1) , \ Me^{\rho s} +a e^{\lambda_2's}\ < \phi(s), \ s \in [T_1-h,T_1). 
$$
Indeed, we can first determine (large negative) $\bar T_1$ as the leftmost root of equation $\phi(t) = Me^{\rho t}$ (with $M$ 
large and positive).  This corresponds to  the limit case $a=0$.  The inequality $\phi'_+(\bar T_1-) >  \phi'_+(\bar T_1+)$ is obvious in such a case. To prove the second inequality, suppose that for a moment that, for some $S \in [\bar T_1-h,\bar T_1)$, 
$$Me^{\rho S} = \phi(S), \ \quad Me^{\rho t} < \phi(t), \ t \in (S,\bar T_1).$$ 
Then $\rho Me^{\rho S} \leq  \phi'(S)$ so that (assuming that  $M$ is large) 
$$
\rho \leq \phi'(S)/\phi(S) \approx \lambda_2, 
$$
a contradiction. Since $a \ll 1$ can considered as a small perturbation parameter, we deduce that the mentioned  
properties hold for all small $a$ (where $T_1$ is  close to $\bar T_1$).

Let $\sigma >0$ be such that $\gamma:= \max \{g'(s): s \in [\kappa - \sigma, \kappa]\} <1$.  From now on, we  fix $a, M, T_1$ chosen above and  take $\epsilon >0, \ c_*-c' >0$ small enough to satisfy  
$$
\phi'_+(T_1+)- \phi'_+(T_1-) < \epsilon \frac{c_* - \sqrt{c_*^2+4}}{2} <0,\  \  -\epsilon(1-\gamma) +(1+ \gamma h) \max_{s \in \R} \phi'(s) (c_*-c') <0,
$$
If $t \in[T_1, T_1+c'h]$, then 
$$
E_+(t,\epsilon,c') =
  (c_*-c')\phi'(t) -  \epsilon + g(\phi_+(t-c'h)) - g(\phi(t-c_*h)). 
$$
Next, for $t \in  [T_1+c'h,T_2],$ we have
$$
E_+(t,\epsilon,c')= \phi''(t) - c'\phi'(t) - \phi(t) - \epsilon + g(\phi(t-c'h)+\epsilon) =  
$$
$$
  (c_*-c')\phi'(t) -  \epsilon + g(\phi(t-c'h)+\epsilon) - g(\phi(t-c_*h)). 
$$
Let us  define $T_1^+$ from 
$$
\phi(T_1^+-2c^*h)= \kappa - \sigma. 
$$
Observe that $T_1^+$ does not depend on $\epsilon, c'$ (thus we may assume that $T_1^+ < T_2$) and that, for some $\theta_1 \in [\kappa - \sigma, \kappa] $ and $\theta_2  >  T_1^+-2c^*h$, 
$$
-\epsilon + g(\phi(t-c'h)+\epsilon) - g(\phi(t-c_*h)) = g'(\theta_1) (\epsilon +\phi'(\theta_2)(c_*-c')h)  -\epsilon \leq 
$$
$$
-\epsilon(1-\gamma) + \gamma h \max_{s \in \R} \phi'(s) (c_*-c'), \quad t \in [T_1^+,T_2].  
$$
As a consequence, we obtain that $E_+(t,\epsilon,c')< 0$ for all $t \in [T_1^+, T_2]$.
On the other hand, if $t \geq T_2+c'h$ then  $E_+(t,\epsilon,c') =0$, and if $t \in [T_2,T_2+c'h)$, it holds 
$$
E_+(t,\epsilon,c') = -\kappa +  g(\phi(t-c'h)+\epsilon) <0.$$
Hence, if $c'$ is close to $c_*$ we find that 
\begin{equation}\label{E+}
E_+(t,\epsilon,c') \leq 0, \ t \in \R\setminus [T_1,T_1^+],  \
\sup_{t \in [T_1,T_1^+]} E_+(t,\epsilon,c')=\omega(c', \epsilon), \lim\limits_{(c' , \epsilon) \to (c_*,0)}\omega(c', \epsilon)\leq 0. 
\end{equation}
Next, Lemma \ref{imp} assures that 
\begin{eqnarray}\label{pfo}&& \hspace{-5mm}\nonumber \phi_+(t)=\frac{1}{\xi'_2
- \xi'_1} \left(\int^t_{-\infty} e^{\xi'_1 (t-s)}g(\phi_+(s-c'h))ds +
\int_t^{+\infty}e^{\xi'_2 (t-s)}g(\phi_+(s-c'h))ds\right. \\ &&  \nonumber
\left.-\int^t_{-\infty} e^{\xi'_1 (t-s)}E_+(s)ds 
-\int_t^{+\infty}e^{\xi'_2 (t-s)}E_+(s)ds
\right) \\ &&  \nonumber
 +\frac{1}{\xi'_2 - \xi'_1}\left[\sum_{t<T_j}e^{\xi'_2(t-T_j)}(\xi'_1\alpha_j
 -\beta_j)+\sum_{t>T_j}e^{\xi'_1(t-T_j)}(\xi'_2\alpha_j-\beta_j)\right], \end{eqnarray}
where  $\beta_2 <0, \ \alpha_2 =0, \ \alpha_1 =\epsilon$, and  $\beta_1 <0$ does not depend on $\epsilon$.  Consider 
\begin{eqnarray*} && {\mathcal E}(t): = -\int^t_{-\infty} e^{\xi'_1 (t-s)}E_+(s)ds 
-\int_t^{+\infty}e^{\xi'_2 (t-s)}E_+(s)ds+ \\ && \sum_{t<T_j}e^{\xi_2(t-T_j)}(\xi'_1\alpha_j
 -\beta_j)+\sum_{t>T_j}e^{\xi'_1(t-T_j)}(\xi'_2\alpha_j-\beta_j).
\end{eqnarray*} Since there exists $\nu >0$ (independent on small $\epsilon, c_*-c'$)  such that $\xi'_1\alpha_1
 -\beta_1 > \nu,$ $ \xi'_2\alpha_1-\beta_1 > \nu, $
we infer from (\ref{E+})  that, for $t \leq T_1$ and small positive $\epsilon, c_* -c'$, 
$$
{\mathcal E}(t) > -\int_{T_1}^{T_1^+}e^{\xi'_2 (t-s)}E_+(s)ds + e^{\xi'_2(t-T_1)}\nu \geq e^{\xi'_2(t-T_1)}(\nu - \frac{\omega(c', \epsilon)}{\xi_2'}) >0. 
$$
Next, for $t \in [T_1, T_1^+]$  and small positive $\epsilon, c_* -c'$, we have that 
$$
{\mathcal E}(t) > -\int^t_{T_1} e^{\xi'_1 (t-s)}E_+(s)ds 
-\int_t^{T_1^+}e^{\xi'_2 (t-s)}E_+(s)ds + e^{\xi'_1(t-T_1)}\nu \geq 
$$
$$
e^{\xi'_1(T_1^+-T_1)}\nu- \omega(c', \epsilon) \sqrt{(c')^2+4}>0. 
$$
Similarly, if $t \geq T_1^+$ then 
$$
{\mathcal E}(t) > -\int_{T_1}^{T_1^+}e^{\xi'_1 (t-s)}E_+(s)ds + e^{\xi'_1(t-T_1)}\nu = e^{\xi'_1(t-T_1^+)}( e^{\xi'_1(T_1^+-T_1)}\nu - \frac{\omega(c', \epsilon)}{|\xi'_1|}) >0. 
$$
Therefore, for all $t \in \R$ and small $c_*-c', \epsilon>0$,  
$$
\phi_+(t)> \frac{1}{\xi'_2
- \xi'_1} \left(\int^t_{-\infty} e^{\xi'_1 (t-s)}g(\phi_+(s-c'h))ds +
\int_t^{+\infty}e^{\xi'_2 (t-s)}g(\phi_+(s-c'h))ds\right). 
$$
To finalize the proof of Theorem \ref{main3}, it suffices now  to repeat Steps III and IV of Section 2. 
The construction of an lower solution is possible because of $c_* > c_\#$:  this inequality assures the existence of two positive real roots $\lambda_2(c') < \lambda_1(c')$ for all $c'$ close to $c_*$. 

\section*{Acknowledgments}  This research was supported by FONDECYT (Chile), pro\-jects  	1080034 (E.T. and M.P.),  
1110309 (S.T.) S. Trofimchuk was also partially supported by CONICYT 
(Chile) through PBCT program ACT-56.

\medskip

\medskip

\end{document}